\documentclass[12pt,a4paper]{article}
\usepackage[utf8]{inputenc}
\usepackage[T1]{fontenc}
\usepackage[british]{babel}
\usepackage{amsmath,amssymb,amsthm}

\usepackage{url}
\usepackage{xcolor}
\usepackage[all]{xy}
\usepackage{hyperref}

\newcommand\invnu{\mathcal{V}}

\newcommand{\abcr}{\operatorname{cr}^{\mathfrak{A}}}

\DeclareMathOperator{\Aut}{Aut}
\DeclareMathOperator{\Cent}{C}

\DeclareMathOperator{\centre}{Z}
\DeclareMathOperator{\cohom}{H}

\DeclareMathOperator{\diag}{diag}
\DeclareMathOperator{\dg}{d}

\DeclareMathOperator{\End}{End}

\DeclareMathOperator{\GL}{GL}
\DeclareMathOperator{\indmin}{i}

\DeclareMathOperator{\maxi}{m}

\DeclareMathOperator{\OLandau}{O}

\DeclareMathOperator{\Orth}{O}

\DeclareMathOperator{\PSL}{PSL}
\DeclareMathOperator{\PSp}{PSp}

\DeclareMathOperator{\rk}{r}
\DeclareMathOperator{\rkm}{rm}
\DeclareMathOperator{\rko}{ro}

\DeclareMathOperator{\rks}{rs}

\DeclareMathOperator{\Soc}{Soc}

\DeclareMathOperator{\Sym}{Sym}

\newtheorem{theorem}{Theorem}[section]
\newtheorem{theorema}{Theorem}

\newtheorem{lemma}[theorem]{Lemma}

\newtheorem{corollary}[theorem]{Corollary}
\theoremstyle{definition}
\newtheorem{example}[theorem]{Example}
\newtheorem{definition}[theorem]{Definition}

\newtheorem{construction}[theorem]{Construction}
\theoremstyle{remark}
\newtheorem{remark}[theorem]{Remark}

\title{Bounds on the number of maximal subgroups %
  of finite groups}
\author{A. Ballester-Bolinches\thanks{Department of Mathematics, Guandong University of Education, Guangzhou, Guangdong, 510303, People's Republic of China}\ \thanks{Departament de Matem\`atiques,
    Universitat de Val\`encia, Dr.\ Moliner, 50, 46100 Burjassot,
    Val\`encia, Spain, \texttt{Adolfo.Ballester@uv.es}, \texttt{Ramon.Esteban@uv.es}. ORCID 0000-0002-2051-9075, 0000-0002-2321-8139} \and R. Esteban-Romero\addtocounter{footnote}{-1}\footnotemark \and P. Jim\'enez-Seral\thanks{Departamento de
    Matem\'aticas, Universidad de Zaragoza, Pedro Cerbuna, 12, 50009 Zaragoza, Spain,
    \texttt{paz@unizar.es}. ORCID 0000-0003-4809-1784}}
\date{}
\begin{document}

\maketitle
\begin{abstract}
  In this paper we obtain significant bounds for the number of
  maximal subgroups 
  of a given index of a finite group. %
  These results allow us to give new
  bounds for the number of random generators needed to generate a
  finite $d$-generated group with high probability.%


  \emph{Mathematics Subject Clasification (2020):} 20P05, 20E07, 20E28
\end{abstract}
\section{Introduction}
All groups considered in this paper will be finite.

The fact that an $r$-tuple of elements of a group $G$ does not generate $G$ is equivalent to the fact that all these elements belong to a certain maximal subgroup $M$ of~$G$. The probability that a certain $r$-tuple of elements of $G$ chosen at random and with a uniform probability distribution is contained in a maximal subgroup $M$ of~$G$ is $1/\lvert G:M\rvert^r$. This makes interesting to consider the number $\maxi_n(G)$ of
maximal subgroups of a group $G$ of index~$n$ for a natural number~$n$ in problems related with the probability that a certain tuple of elements of the group generates it.

Pak \cite{Pak99-prob}, motivated by potential
applications for the product replacement algorithm, widely used to
generate random elements in a finitely generated group, analysed an invariant related to the probability a tuple generates a group.
\begin{definition}
  Given a group $G$, we denote by $\invnu(G)$ the least positive
  integer $k$ such that $G$ is generated by $k$ random elements with
  probability at least  $1/\mathrm{e}$.
\end{definition}
Pak conjectured that
$\invnu(G)={\OLandau(\dg(G)\log\log\lvert G\rvert)}$, where $\dg(G)$ denotes the smallest cardinal of a generating set for~$G$. Here and in all this paper,
the symbol $\log$ will be used to denote the logarithm to the
base~$2$, and we follow the convention that $\log 0=-\infty$.

Given a maximal subgroup $M$ of a group $G$, the quotient $G/M_G$ of
$G$ by the normal core $M_G$ of $M$ in $G$ is a primitive group. According
to a well-known theorem of Baer \cite{Baer57} (see also
\cite[Theorem~1.1.7]{BallesterEzquerro06}), $G/M_G$ is of one of three
different
types. In this case, we say that the maximal $M$ is of the same type as
the primitive group $G/M_G$ it induces, like in
\cite[Definition~1.1.19]{BallesterEzquerro06}. Lubotzky \cite{Lubotzky02} established bounds on the number of maximal subgroups of type~1 of a given index~$n$ and on the number of maximal subgroups of types~2 and~3 of index~$n$ by considering the number of abelian and non-abelian chief factors, respectively, in a given chief series of the group. He obtained the following result that confirms the validity of Pak's conjecture.
\begin{theorem}[Lubotzky, {\cite[Corollary~2.6]{Lubotzky02}}]\label{th-Lubotzky-mn}
  If $G$ is a group with $r$ chief factors in a given chief
  series, $r_a$ of them abelian and $r_b$ of them non-abelian, then 
  \[\maxi_n(G)\le \bigl(\tfrac{1}{2}(r_b+1)r_b+r_an^{\dg(G)}\bigl)n^2\le r^2 n^{\dg(G)+2}%
    .\]
  Furthermore,
  \[\invnu(G)\le \frac{1+\log\log\lvert G\rvert}{\log
      {\indmin(G)}}+\max\left(\dg(G), \frac{\log\log\lvert G\rvert}{\log
        {\indmin(G)}}\right)+2.02,\]
  where $\indmin(G)$ denotes the smallest index of a proper subgroup of~$G$.
\end{theorem}

The bound for $\invnu(G)$ in Theorem~\ref{th-Lubotzky-mn} depends on the following invariant.
\begin{definition}
  For a group $G$, let us call 
  \[\mathcal{M}(G)=\max_{n\ge 2}\log_n\maxi_n(G)=\max_{n\ge 2} \frac{\log \maxi_n(G)}{\log
    n}.\]
\end{definition}
Here we use $\log_nx$ to denote the logarithm to the base $n$ of
$x$, that is, $\log_nx=\log x/\log n=\ln x/\ln n$.
There is a close relation between $\mathcal{M}(G)$ and $\invnu(G)$.
\begin{theorem}[Lubotzky, {\cite[Proposition~1.2]{Lubotzky02}}]\label{th-Lubotzky-M}
  \[\mathcal{M}(G)-3.5\le \invnu(G)\le \mathcal{M}(G)+2.02.\]
\end{theorem}

The proof of the left hand side bound in the above result depends on the following
result.%

\begin{theorem}[see
  {\cite[Theorem~1.3]{Lubotzky02} or \cite[Theorem~21]{DetomiLucchini06-noncommalggeom}}]\label{th-b-2} %
  There exists a constant $b$ such that for every group $G$ and
  every $n\ge 2$, $G$ has at most $n^b$ core-free maximal subgroups of
  index~$n$. In fact, $b=2$ will do.
\end{theorem}

Unless otherwise stated, we will follow the notation of the books
\cite{DoerkHawkes92} and \cite{BallesterEzquerro06}. Detailed
information about primitive groups and  chief factors, crowns, and
precrowns of a group can be found in \cite[Chapter~1]{BallesterEzquerro06}.

Detomi and Lucchini
\cite{DetomiLucchini03}, with the help of the number  $\lambda(G)$ of non-Frattini
chief factors in a given chief series of~$G$ and by considering their
associate crowns,  proved independently also the
validity of Pak's
conjecture.

\begin{theorem}[Detomi and Lucchini {\cite[Theorem~20]{DetomiLucchini03}}]\label{th-DetomiLucchini03-20}
  There exists a constant $c$ such that, for any group $G$,
  $\invnu(G)\le \lfloor\dg(G)+c\log\lambda(G)\rfloor$ if $\lambda(G)>1$,
  otherwise, $\invnu(G)\le \lfloor \dg(G)+c\rfloor$, where $\lambda(G)$
  denotes the number of non-Frattini chief factors in a given chief
  series of~$G$.
\end{theorem}
Here the symbol 
$\lfloor x\rfloor$ denotes the defect integer part of $x$, that
is, the largest integer number $n$ such
that $n\le x$. %

The proof of this result uses some results of Dalla Volta and Lucchini
\cite{DallaVoltaLucchini98} and Dalla Volta, Lucchini, and Morini
\cite{DallaVoltaLucchiniMorini01} about the number of generators of powers
of certain diagonal-type subgroups of direct powers of copies of a
monolithic group.

The aim of this paper is to obtain
tighter bounds for $\maxi_n(G)$, and so for $\invnu(G)$, by considering the numbers of maximal
subgroups of each type, as in Lubotzky's paper \cite{Lubotzky02}, and with the help of the crowns associated to abelian chief factors, like in Detomi and Lucchini's paper \cite{DetomiLucchini03}, and to non-abelian chief factors. Our bounds will depend on four invariants. Our first invariant is related to maximal
subgroups of type~1.

\begin{definition}
  Let $G$ be a group and let $n>1$ be a natural number. We denote by
  $\abcr_n(G)$ the number of crowns
  associated to complemented abelian chief factors of order~$n$
  of~$G$, that is, the number of $G$-isomorphism classes of
  complemented abelian chief factors of~$G$.
\end{definition}
Clearly, $\abcr_n(G)=0$ unless $n$ is a power of a prime.

In order to define our second invariant, which concerns
non-abelian chief factors and is related to the primitive quotients of type~2, we must recall the next definition.
\begin{definition}[see {\cite[Definition~1.2.9]{BallesterEzquerro06}}]\label{def-prim-gp-chieff}
  Given a chief factor $H/K$ of a group $G$, the \emph{primitive group
  $[H/K]*G$ associated with $H/K$ in $G$} coincides with the
semidirect product $[H/K]\bigl(G/{\Cent_G(H/K)}\bigr)$ if $H/K$ is
abelian and with the quotient group $G/{\Cent_G(H/K)}$ if $H/K$ is
non-abelian. 
\end{definition}

\begin{definition}
  Let $n$ be a natural number. The symbol ${\rks_n(G)}$ denotes the
  number of non-abelian chief factors $A$ in a given chief series of $G$ such that the
  associated primitive group $[A]*G$ has a core-free maximal subgroup of
  index~$n$.
\end{definition}

Our third and fourth invariants concern also non-abelian chief factors and
contain information about the primitive quotients of type~3 of the group. 
\begin{definition}
  Let $n$ be a natural number. The symbol ${\rko_n(G)}$ denotes the
  number of non-abelian chief factors $A$ in a given chief series of $G$ such that $A$
  has order~$n$.
\end{definition}

For the definition of our fourth invariant, the following auxiliary invariant is useful.
\begin{definition}
  Let $G$ be a group. For a characteristically simple
  group $A$, that is, a direct product of copies of a
  simple group $S$,
  $\rk_A(G)$ denotes the number of non-Frattini chief factors isomorphic to~$A$ in a given chief series of~$G$.
\end{definition}

According to Theorem~A of
\cite{BallesterEstebanJimenez21-largecharsim}, $\rk_A(G)$ coincides with the 
largest number $r$ such that $G$ has a normal
section that is the direct product of $r$ non-Frattini chief factors of $G$
that are isomorphic (not necessarily $G$-isomorphic) to~$A$.

\begin{definition}
  Let $n$ be a natural number. The symbol ${\rkm_n(G)}$ denotes the maximum of the lenghts of the $G$-crowns associated to non-abelian chief factors of order~$n$ of~$G$.
\end{definition}

We will see that, according to Theorem~\ref{th-power-order}, the
non-abelian chief factors $A$ of $G$ of order~$n$ fall into at most
two isomorphism classes. We also note that a non-abelian chief factor can be counted in $\rks_n(G)$ for different values of $n$, if the corresponding primitive group has core-free maximal subgroups of different indices. Furthermore, a non-abelian chief factor of order~$n$ that is counted in $\rko_n(G)$ can also be counted in $\rks_n(G)$ if the primitive group associated to the corresponding chief factor has a maximal subgroup complementing its socle (see \cite{JimenezLafuente98} or \cite[Theorem~1.1.48]{BallesterEzquerro06}),

We obtain a bound for $\maxi_n(G)$ that improves the one of Theorem~\ref{th-Lubotzky-mn}. 
\begin{theorema}\label{th-Lub-A}
  Let $G$ be a non-cyclic group with $r$ chief factors in a given chief series. For every natural $n\ge 2$, $\maxi_n(G)\le rn^{\dg(G)+2}$.
\end{theorema}
This bound will appear as a consequence of a tighter bound in terms of the invariants $\abcr_n(G)$, $\rks_n(G)$, $\rkm_n(G)$, and $\rko_n(G)$ presented in Theorem~\ref{th-mn-long}.
We also obtain a bound for $\invnu(G)$.
\begin{theorema}\label{th-A}
  Let $G$ be a $d$-generated non-trivial group with $d\ge 2$. Then
  \[%
  \invnu(G)\le \eta(G),\]%
where %
$\eta(G)$
is a function bounded by a linear combination of $d$ and the maxima of
$\log_n\abcr_n(G)$, $\log_n\rks_n(G)$, $\log_n\rko_n(G)$, and
$\log_n\rkm_n(G)$.
\end{theorema}

The precise value of the function $\eta(G)$ will be presented in detail in
Theorem~\ref{th-A-detailed}. %

We %
obtain upper bounds for the number of maximal subgroups of a given
index in a group in Section~\ref{sect-upper-bounds}. 
We show that the value of the constant in the bounds of %
Detomi and Lucchini
\cite{DetomiLucchini03} (Theorem~\ref{th-DetomiLucchini03-20}) can be computed
from Theorem~\ref{th-A}. This will be done in Section~\ref{sect-comp}.
 We also present a general procedure
(Construction~\ref{const-hat}) to obtain, starting from a
$d$-generated primitive group $H$ of type~1, a $d$-generated group $G$
with all possible $G$-isomorphism classes of chief factors isomorphic
to the socle of $H$ and with associated primitive groups isomorphic
to~$H$. Similar ideas (see Remark~\ref{rem-many-crowns}) can be used
to create a $d$-generated group with all possible crowns associated to non-Frattini
abelian chief factors of a given order whose associated primitive
groups are $d$-generated.

\section{Upper bounds for the number of maximal subgroups of a given
  index in a group}\label{sect-upper-bounds}

\subsection{Bounds for the number of maximal subgroups of type~1}
Let $L$ be a non-cyclic group with a unique complemented abelian minimal normal
subgroup $N$. Assume also that $N$ is complemented in $L$. As in
\cite{DallaVoltaLucchini98}, let $L^k$
be the direct product of $k$ copies of $L$ and set
\[L_k=\{(l_1,\dots, l_k)\in L^k\mid l_1\equiv\dots\equiv
l_k\pmod{N}\}.\]
Then $L_k$ is the product of $N^k$ and a diagonal subgroup
\[\diag L^k=\{(l,\dots, l)\in L^k\mid l\in L\}\]
of $L^k$. The socle of $L_k$ is a direct product of $k$ minimal normal
subgroups of $L_k$, each of them isomorphic to $N$, and that
$L_k/N^k\cong L/N$. The quotient of $L_{k+1}$ over a minimal normal
subgroup of $L_{k+1}$ is isomorphic to $L_k$. 
We have that $\dg(L_1)\le
\dg(L_2)\le \dg(L_3)\le\dotsb$ and, according to \cite{Lucchini95},
$\dg(L_{k+1})\le \dg(L_k)+1$ for every $k$. Hence, if $d\ge \dg(L)$, then
there is a unique $k$ such that $\dg(L_k)=d<\dg(L_{k+1})$. Let us call
$f(d)=k+1$. Then, by \cite{DallaVoltaLucchini98}, 
\[f(d)=1+\log_q\bigl(\lvert N\rvert^{d-1}\big/\lvert
\cohom^1(L/N,N)\rvert\bigr),\]
where $q=\lvert {\End_{L/N}(N)}\rvert$.

\begin{theorem}\label{th-nd-abelian-crown}
  Let $U$ be a maximal subgroup of type~1 of a $d$-generated group $G$ and let $n=\lvert G:U\rvert$. If $U$ is not normal in $G$, then the
  number of maximal subgroups $M$ of $G$ such that $\Soc(G/M_G)$ is
  $G$-isomorphic to $\Soc(G/U_G)$ is less than or equal to 
  \[\frac{n^{d}-n\lvert \cohom^1(G/C, A)\rvert}{q-1},\]
  where $A=C/U_G$ is the unique minimal normal subgroup of $G/U_G$ and 
  $q=\lvert\End_{G/C}(A)\rvert$. If $U$ is normal in $G$, then the number of maximal subgroups $M$ of $G$ such that $\Soc(G/M_G)$ is $G$-isomorphic to $\Soc(G/U_G)$ is less than or equal to
  \[\frac{n^d-1}{n-1}.\]
  In all cases, this number is less
  than or equal to $n^d-1$.
\end{theorem}
\begin{proof}
 Consider the crown $C/R$ associated to $A$, Observe that
  $C/U_G=\Soc(G/U_G)$. If
  \[L=[A]*G=[A]\bigl(G/{\Cent_G(C/U_G)}\bigr)=[A](G/C)\cong G/U_G\] is the
  primitive group associated to $A$ and $C/R$ is the direct product of
  $k$ chief factors isomorphic to $A$, then $G/R\cong L_k$.

    Let us suppose first that $U$ is not normal.  Note that
  $\dg(L_k)\le d=\dg(G)$.  By the comments after \cite[Theorem~2.7]{DallaVoltaLucchini98}, we have that
  \begin{equation}
    k \le \log_q\frac{n^{d-1}}{\lvert\cohom^1(L/A,A)\rvert}.
    \label{eq-DVL-2.7-b}
  \end{equation}
  
  Let $M$ be a maximal subgroup of $G$ such that $\Soc(G/M_G)$ is  $G$-isomorphic to $A$. We have that $\Soc(G/M_G)=C/M_G$ and $R\le M_G\le C$. Moreover, there exists an isomorphism $\varphi\colon G/M_G\longrightarrow L$ satisfying that the composition with the projection over $G/C$, $G/M_G\longrightarrow L\longrightarrow G/C$, has kernel $C$ and gives he identity on $G/C$. Let us now give a bound for the number of possible kernels of epimorphisms from $L_K$ to $L$ giving the identity on $G/C$. By \cite[Lemma~2.5]{DallaVoltaLucchini98}, the number of possible kernels of epimorphisms from $L_k$ onto $L$ giving the identity on $G/C$ is $(q^k-1)/(q-1)$. By Equation~\eqref{eq-DVL-2.7-b},
  \[\frac{q^k-1}{q-1}\le \frac{\frac{n^{d-1}}{\lvert \cohom^1(L/A,A)\rvert}-1}{q-1}=\frac{n^{d-1}-\lvert \cohom^1(L/A,A)\rvert}{(q-1)\lvert \cohom^1(L/A, A)\rvert}.\]

The number of maximal subgroups of type~1 with a given core coincides with the number of complements of the minimal normal subgroup of the associated primitive group, namely $n\lvert \cohom^1(L/A, A)\rvert$. Hence the number of maximal subgroups $M$ of $G$ such that $\Soc(G/M_G)$ is $G$-isomorphic to $A=\Soc(L)$ is less than or equal to
  \[\frac{n^{d-1}-\lvert\cohom^1(L/A, A)\rvert}{(q-1)\lvert \cohom^1(L/A, A)\rvert}\cdot n\lvert \cohom^1(L/A, A)\rvert=\frac{n^d-n\lvert \cohom^1(L/A,A)\rvert}{q-1}.\]

  Suppose now that $A$ is central, that is, that $U$ is normal in
  $G=C$, then $n$ is a prime and $G/R=C/R$ is a direct product of $k$
  cyclic groups isomorphic to $C_n$. Since $G$ is $d$-generated, $k\le
  d$. In this case, $q=n$ is a prime number and the number of normal subgroups $M$ of index~$n$ such that $R\le M\le G$ is
  \[\frac{n^k-1}{n-1}\le \frac{n^d-1}{n-1}.\qedhere\]
\end{proof}

\begin{remark}
  The bound of Theorem~\ref{th-nd-abelian-crown} is attained in groups
  like $L_k$, where $L\cong \Sym(4)$. In this group, $q=2$, $n=4$,
  and 
  $\lvert\cohom^1(G/C, A)\rvert=1$. If $k=f(d)-1=2d-2$, for $d\ge 2$, we
  obtain that the number of maximal subgroups of index $4$ of $L_k$ is
  exactly 
  \[\frac{4^d-4}{2-1}=4^d-4.\]
\end{remark}

\begin{corollary}\label{corol-max-type-1}
  The number of maximal subgroups $M$ of type~1 and index $n=p^r$ of a
  $d$-generated group $G$ is less than or equal to $(n^{d}-1){\abcr_n(G)}$.
\end{corollary}
\begin{proof}
  Given a
  maximal subgroup $U$ of $G$ of type~1 and index $n$, we have that its socle $\Soc(G/U_G)$
  must fall into one of the $G$-isomorphism classes of chief
  factors. %
  By
  Theorem~\ref{th-nd-abelian-crown}, we obtain that the number of
  maximal subgroups $M$ of $G$ such that $\Soc(G/M_G)\cong_G
  \Soc(G/U_G)$ is bounded
  by $n^{d}-1$. Since the number of $G$-isomorphism classes of chief
  factors of order $n$ is $\abcr_n(G)$, we obtain that the number of maximal
  subgroups of type~1 and index $n$ of $G$
  must be bounded by $(n^{d}-1){\abcr_n(G)}$.
\end{proof}

\subsection{Bounds for the number of maximal subgroups of type~2}

The next result provides a bound for the number of maximal subgroups
of type~2 of a given index of a group.
\begin{theorem}\label{th-upper-max-2}
  Let $G$ be a group and let $n$ be a natural number. The number of
  maximal subgroups of $G$ of type~2 and index~$n$ is bounded by
  ${\rks_n(G)}n^2$.
\end{theorem}
\begin{proof}
  Let $M$ and $U$ be two maximal subgroups of $G$ of type~2 and
  index~$n$,  By \cite[Proposition~1.2.4]{BallesterEzquerro06}, $M_G=\Cent_G\bigl({\Soc(G/M_G)}\bigr)$ and
  $U_G=\Cent_G\bigl({\Soc(G/U_G)}\bigr)$ coincide if, and only if,
  $\Soc(G/M_G)$ and $\Soc(G/U_G)$ are $G$-isomorphic. Hence by the
  Jordan-H\"older theorem, if $M_G\ne U_G$, then in a given chief series of $G$ there exists
  a chief factor $G$-isomorphic to $\Soc(G/M_G)$ and another chief
  factor $G$-isomorphic to $\Soc(G/U_G)$. Consequently the number of
  primitive quotients of $G$ associated to a maximal subgroup of
  index~$n$ is $\rks_n(G)$. By Theorem~\ref{th-b-2}, the number
  of core-free maximal subgroups of index~$n$ of $G/U_G$ for a 
  maximal subgroup $U$ of type~2 of~$G$ is bounded by
  $n^2$, that is, the number of maximal subgroups of index~$n$ of $G$
  with core $U_G$ is bounded by $n^2$. Hence the number of maximal subgroups of $G$ of type~2 and index~$n$ 
  is bounded by ${\rks_n(G)}n^2$.
\end{proof}

\subsection{Bounds for the number of maximal subgroups of type~3}
Now we analyse the maximal subgroups of type~3. The following two results show that the minimal normal subgroups of primitive groups of type~3 fall in at most two isomorphism classes. Although this is not needed for the proofs of our results, we include this information here for the sake of completeness. The first result appears as a consequence of the classification of
simple groups.
\begin{lemma}\label{lemma-2-same-order}
  The number of non-abelian simple groups of a given order is
  at most~$2$.
\end{lemma}

A generalisation of this result to characteristically simple groups was formulated as a question by
Cameron \cite{Cameron81}, who attributed its proof to Teague in
Note~(ii) at the end of his paper.
\begin{theorem}[see
  {\cite[Theorem~6.1]{KimmerleLyonsSandlingTeague90}}]\label{th-power-order}
  Let $S$ and $T$ be non-isomorphic finite simple groups. If $\lvert
  S^a\rvert = \lvert T^b\rvert$ for some natural numbers $a$ and $b$,
  then $a=b$ and $S$ and $T$ are either $\PSL_3(4)$ or $\PSL_4(2)$, or
  are $\Orth_{2n+1}(q)$ and $\PSp_{2n}(q)$ for some $n\ge 3$ and some
  odd~$q$.
\end{theorem}

As in
\cite{JimenezLafuente98}, we use the non-abelian cohomology with the
terminology of Serre in~\cite{Serre64}. Given a $G$-group $B$ and a
$1$-cocycle $\beta\in\centre^1(G,B)$, then $b^{\eta(g)}=b^{gg^\beta}$
defines a homomorphism $\eta\colon G\longrightarrow \Aut B$. The
corresponding $G$-group is denoted by $B_\beta$ and called the
$G$-group obtained from $B$ by torsion via~$\beta$.

\begin{lemma}
  Let $G$ be a $d$-generated group.
  Let $n$ be a natural number which is a power of the order of a non-abelian simple group. Then %
  $\rkm_n(G)\le n^d$.
\end{lemma}
\begin{proof}
  Consider a crown with a chief factor $A$ of order $n$ and let  $B$ be another chief factor in the same crown as~$A$. We have that $A$ and $B$ are $G$-connected (see \cite[Definition~1.2.5]{BallesterEzquerro06}), but not $G$-isomorphic. According to \cite[Proposition~1.2]{JimenezLafuente98}, there exists a $1$-cocycle $\beta\in\centre^1(G,B)$ such that $A$ is isomorphic to $B_\beta$ as $G$-groups, that is, there exists an isomorphism $\varphi\colon A\longrightarrow B$ such that $a^{g\varphi}=a^{\varphi g g^\beta}$ for each $a\in A$ and $g\in G$. Hence the number of chief factors $G$-connected to $B$ but not $G$-isomorphic to $B$ is bounded by $\lvert \centre^1(G, B)\rvert$ and so $\rkm_n(G)\le \lvert \centre^1(G,B)\rvert$. Since these $1$-cocycles are uniquely determined by the image of a generating system of~$G$, we have that the number of cocycles is at most $n^d$. %
  It follows that $\rkm_n(G)\le n^d$.
\end{proof}
\begin{theorem}\label{th-upper-max-3}
  Let $G$ be a $d$-generated group and let $n$ be a natural number that is a power of the order of a non-abelian simple group. Assume that $G$ has $s_n$ crowns of (non-abelian) chief factors of order~$n$ that are the products of $t_1$, $t_2,\dots$, $t_{s_n}$ chief factors, respectively.  The number of maximal subgroups of~$G$ of type~3 and index~$n$ is bounded by
  \[n^2\sum_{j=1}^{s_n}\frac{t_j(t_j-1)}{2}.
  \]
\end{theorem}
\begin{proof}
  Given a maximal subgroup $U$ of $G$ of type~3 and index~$n$, the quotient $G/U_G$ is a primitive group of type~$3$.

  The socle of a primitive quotient of type~3 of $G$ is a direct product of two minimal normal subgroups that are $G$-isomorphic to two $G$-connected chief factors of $G$ belonging to the same $G$-crown. Since the core of this maximal subgroup is the intersection of the centralisers of the minimal normal subgroups of the quotient, that are the centralisers of the chief factors appearing in the crown, we have that the number of choices for the core is $t_i(t_i-1)/2$, where $t_i$ is the number of chief factors in the corresponding crown. This shows that the number of cores of such maximal subgroups is bounded by $\sum_{j=1}^{s_n}t_j(t_j-1)/2$. By Theorem~\ref{th-b-2}, each of the quotients by these cores has at most $n^2$  core-free maximal subgroups of index~$n$. The result follows.
\end{proof}
\begin{corollary}\label{corol-upper-max-3-bis}
  Let $G$ be a $d$-generated group and let $n$ be a natural number
  which is a power of the order of a non-abelian simple group. The number of
  maximal subgroups of $G$ of type~3 and index~$n$ is bounded by
  \[n^2\left(\frac{\rkm_n(G)(\rko_n(G)-s_n)}{2}\right)\le n^{d+2}\left(\frac{\rko_n(G)-s_n}{2}\right),\]
  where $s_n$ is the number of $G$-crowns of $G$ associated to non-abelian chief factors.
\end{corollary}
\begin{proof}
  Since, with the notation of Theorem~\ref{th-upper-max-3}, we have that $\rko_n(G)=\sum_{i= 1}^{s_n}t_i$ and $\rkm_n(G)=\max_{i=1}^{s_n}t_i$, this result is an immediate consequence of Theorem~\ref{th-upper-max-3}.
\end{proof}

\subsection{General bounds}

Denote by $\mathbb{T}$ the set of all prime powers greater than $1$
and 
by $\mathbb{S}$ the set of all powers of the orders of non-abelian
simple groups.
As a consequence of Corollary~\ref{corol-max-type-1},
Theorem~\ref{th-upper-max-2}, and Theorem~\ref{th-upper-max-3}, we obtain the
following conclusion.
\begin{theorem}\label{th-mn-long}
  The number ${\maxi_n(G)}$ of maximal subgroups of index~$n$ of a $d$-generated  group $G$ satisfies the following bounds:
  \begin{enumerate}
  \item If $n\in\mathbb{T}$, then
    \begin{align*}
      {\maxi_n(G)}&\le
    (n^d-1){{\abcr_n(G)}+n^2{\rks_n(G)}}.
  \end{align*}
  \label{en-mn-long-1}
  \item If $n\in\mathbb{S}$, then
    \begin{align*}
      {\maxi_n(G)}&\le
                    n^2{{\rks_n(G)}}+n^2\left(\frac{\rkm_n(G)\rko_n(G)}{2}\right)\\
      &\le n^2{\rks_n(G)}+n^{d+2}\left(\frac{\rko_n(G)}{2}\right).\\
    \end{align*}
    \label{en-mn-long-2}
  \item If $n\notin\mathbb{S}\cup\mathbb{T}$, then 
    \[{\maxi_n(G)}\le n^2{{\rks_n(G)}}.
    \]\label{en-mn-long-3}
  \end{enumerate}
\end{theorem}

We can use the bound $A+B\le 2\max\{A,B\}$ to obtain the following
bounds for $\maxi_n(G)$:
  \begin{enumerate}
  \item If $n\in\mathbb{T}$, then
    \begin{align*}
      {\maxi_n(G)}&\le 2\max\{n^d{\abcr_n(G)},n^2{\rks_n(G)}\}.
  \end{align*}
  \label{en-mn-long-1p}
  \item If $n\in\mathbb{S}$, then
    \[
    {\maxi_n(G)}\le n^2\max\left\{2\rks_n(G), \rkm_n(G)\rko_n(G)\right\}.
    \]
  \item If $n\notin\mathbb{S}\cup\mathbb{T}$, then 
    \[{\maxi_n(G)}\le n^2{{\rks_n(G)}}.\]\label{en-mn-long-3p}
  \end{enumerate}

Therefore we can obtain the following bound for $\mathcal{M}(G)$.
\begin{align*}
  \mathcal{M}(G)&=\max_n\log_n\maxi_n(G)\\
  &\le \max_n\Bigl\{\max\{d+\log_n2+\log_n\abcr_n(G),
  2+\log_n2+\log_n\rks_n(G)\mid n\in\mathbb{T}\},\\
  &\qquad\qquad 2+\max\bigl\{\log_n2+\log_n\rks_n(G),\\
                &\qquad\qquad\qquad\qquad\!\log_n\rkm_n(G)+\log_n\rko_n(G)\mid
  n\in\mathbb{S}\bigr\},\\
  &\qquad\qquad 2+\max\{\log_n\rks_n(G)\mid n\notin\mathbb{S}\cup
  \mathbb{T}\}\Bigr\}\\
  &\le \max\Bigl\{d+\max_{n\in\mathbb{T}}\{\log_n2+\log_n\abcr_n(G)\},\\
  &\qquad\qquad
  2+\max_n\{\log_n2+\log_n\rks_n(G)\},\\
                &\qquad\qquad 2+\max_{n\in\mathbb{S}}\bigl\{\log_n\rkm_n(G)+\log_n\rko_n(G)\}\Bigr\}\\
  &\le \max\Bigl\{d+\max_{n\in\mathbb{T}}\{\log_n2+\log_n\abcr_n(G)\},\\
  &\qquad\qquad
  2+\max_n\{\log_n2+\log_n\rks_n(G)\},\\
                &\qquad\qquad 2+d+\max_{n\in\mathbb{S}}\bigl\{\log_n\rko_n(G)\}\Bigr\}.
\end{align*}
We are in a position to prove Theorem~\ref{th-Lub-A}.

\begin{proof}[Proof of Theorem~\ref{th-Lub-A}]
  Note that $\abcr_n(G)\le r$, $\rks_n(G)\le r$, $\rko_n(G)\le r$, and $\rkm_n(G)\le n^d$.
  By Theorem~\ref{th-mn-long}, if $n\in\mathbb{T}$, then
  \[\maxi_n(G)\le n^dr+n^2r\le 2n^dr\le n^{d+2}r%
    ;\]
  if $n\in \mathbb{S}$, then
  \[\maxi_n(G)\le n^2r+n^2rn^d/2\le n^{d+2}r%
    ,\]
  and if $n\notin\mathbb{S}\cup\mathbb{T}$, then
  \[\maxi_n(G)\le n^2r\le n^{d+2}r%
    .\qedhere\]

\end{proof}
We can obtain from Theorem~\ref{th-mn-long} the value of the function of the upper bound of
Theorem~\ref{th-A}.
\begin{theorem}\label{th-A-detailed}
  Let $G$ be a $d$-generated non-trivial group. Then, for
  \begin{align*}
  \eta(G)
  &:= \max\Bigl\{d+2.02+\max_{n\in\mathbb{T}}\{\log_n2+\log_n\abcr_n(G)\},\\
  &\qquad\qquad\!
  4.02+\max_n\{\log_n2+\log_n\rks_n(G)\},\\
    &\qquad\qquad\!4.02+\max_{n\in\mathbb{S}}\bigl\{\log_n\rkm_n(G)+\log_n\rko_n(G)\}\Bigr\},
\end{align*}
and
  \begin{align*}
  \kappa(G)
  &:= \max\Bigl\{d+2.02+\max_{n\in\mathbb{T}}\{\log_n2+\log_n\abcr_n(G)\},\\
  &\qquad\qquad\!
  4.02+\max_n\{\log_n2+\log_n\rks_n(G)\},\\
    &\qquad\qquad\!4.02+d+\max_{n\in\mathbb{S}}\bigl\{\log_n\rko_n(G)\}\Bigr\},
\end{align*}
we have that
\[\invnu(G)\le \eta(G)\le \kappa(G).\]
\end{theorem}

\begin{remark}
  The terms $\log_n2$ in  Theorem~\ref{th-A-detailed}
  depend on the bound $A+B\le 2\max\{A, B\}$. If $A=0$ or $B=0$, then
  $A+B=\max\{A,B\}$. In particular, in the computation of the maxima the terms
  $\log_n2$ appear only if $n\in \mathbb{T}$ and
  ${\abcr_n(G)}{\rks_n(G)}\ne 0$, or $n\in\mathbb{S}$ and
  ${\rks_n(G)}{\rko_n(G)}\ne 0$. In all other cases, $\log_n2$ can be
  safely removed to estimate $\invnu(G)$.
\end{remark}

\section{Discussion}\label{sect-comp}

The aim of this section is to show that the value of the constant in the upper bound of
Theorem~\ref{th-DetomiLucchini03-20} %
appears as a 
consequence of Theorem~\ref{th-A} and to present a construction of groups with a large number of crowns of chief factors whose associated primitive quotients are isomorphic to a given primitive group of type~1.

\subsection{Determination of the constant of Theorem~\ref{th-DetomiLucchini03-20}}
We can use the upper bound of Theorem~\ref{th-A} to determine the value of the constant~$c$ of Theorem~\ref{th-DetomiLucchini03-20}.
\begin{theorem}
  The bounds of Theorem~\ref{th-DetomiLucchini03-20} hold with $c=5.02$.
\end{theorem}
\begin{proof}%
  Suppose first that the number $\lambda(G)$ of non-Frattini chief
  factors in a given chief series of~$G$ is $1$. In this case,
  $G/\Phi(G)$ is the only non-Frattini chief factor of~$G$ and
  $G/\Phi(G)$ must be simple, either abelian or non-abelian. If
  $G/\Phi(G)$ is abelian, then $G$ is cyclic and $G$ has exactly one crown
  of non-Frattini abelian chief factors. Therefore $\invnu(G)\le
  d+\log_22+2.02=d+3.02$ (note that the term $\log_22$ can even be
  safely removed from this sum). Assume now that $G/\Phi(G)$ is
  non-abelian. It follows that $\rks_n(G)=1$, where $n$ is the index of a maximal subgroup of~$G/\Phi(G)$
  and so $\invnu(G)\le 4.02+\log_52< 5.02$
  (we could even remove the term $\log_52$). In both cases, $\invnu(G)\le
  \lfloor d+3.02\rfloor$.

  Now assume that $\lambda(G)>1$. It will be enough to show that for a
  constant~$c$, all three elements whose maximum is being considered
  are bounded by $d+c\log \lambda(G)$ for a constant $c$. Note also
  that $\abcr_n(G)$, $\rks_n(G)$, $\rko_n(G)$, and $\rks_n(G)$ are
  less than or equal to $\lambda(G)$. The first term satisfies

  \begin{align*}
    d+2.02+\max_n\{\log_n2+\log_n\abcr_n(G)\}&\le
  d+2.02+1+\log\lambda(G)\\&\le d+4.02\log\lambda(G).
  \end{align*}
  The second term satisfies the inequality
  \begin{align*}
    4.02+\max_n\{\log_n2+\log_n\rks_n(G)\}&\le
    4.02+\log_52+\frac{\log\lambda(G)}{\log5}\\
    &\le
    d+\frac{2.02+\log_52+1}{\log5}\log\lambda(G)\\&\le
    d+2.31\log\lambda(G)
  \end{align*}
  (recall that we can assume $d\ge 2$ because $G$ is not
  cyclic). Finally, the third term satisfies
  \[
    4.02+d+\max_{n\in\mathbb{S}}\bigl\{\log_n\rko_n(G)\}\le 4.02+d+\lambda(G)\]
  Therefore, Theorem~\ref{th-DetomiLucchini03-20} holds with $c=5.02$.
\end{proof}

\subsection{Some examples and constructions}
We conclude this section by presenting a couple of
examples in which our bound improves substantially the previously known bounds.%

In the following we will construct some groups with many crowns of chief factors whose associated primitive quotients are isomorphic to a given monolitic primitive group. This construction will depend on a a general construction for subdirect products.
\begin{construction}\label{const-subd}
  Suppose that $G_i=\langle a_{i1},\dots, a_{id}\rangle$, $1\le i\le r$ are $d$-generated groups. Let $D=\prod_{i=1}^rG_i$ be the direct product of the groups $G_i$, $1\le i\le r$. Consider $a_j=\prod_{i=1}^ra_{ij}\in D$, $1\le j\le d$. Let $\hat G=\langle a_1,\dots,a_r\rangle\le D$. We have that the restriction to $\hat G$ of the natural projection $\pi_i\colon D\longrightarrow G_i$,  $1\le i\le r$, is a group epimorphism with kernel $\hat G\cap \prod_{k\ne i}G_k$. Hence each chief factor $H_i/K_i$ of~$G_i$ is isomorphic to a chief factor $H/K$ of~$\hat G$. Furthermore, $G_i$ acts on $H_i/K_i$ as $\hat G$ acts on $H/K$. Moreover, $\bigcap_{i=1}^r (\hat G\cap \prod_{k\ne i}G_k)=1$, and so $\hat G$ is a subdirect product of the subgroups $G_i$. Consequently, each chief factor of $\hat G$ is isomorphic to a chief factor of one of the $G_i$. However, it might happen that the number of chief factors of $\hat G$ is less than the number of chief factors of all the $G_i$. For instance, this happens if $G_i=\langle a_{1i}\rangle\cong C_p$, $1\le i\le 2$: in this case, $\hat G\cong C_p$.
\end{construction}

The following construction for $d$-generated groups with all possible crowns whose associated primitive quotients are isomorphic to a given monolitic primitive group is motivated by an argument of Hall \cite{Hall36-qjm}.
\begin{construction}\label{const-hat}
    Let $G$ be a $d$-generated group. We can consider the set $\Omega$ of all
  generating ordered $d$-tuples of $G$. The group $\Aut G$ acts
  on~$\Omega$ in the natural way. This decomposes $\Omega$ into orbits
  for the action of $\Aut G$, $\Omega=\Omega_1\cup\dots\cup\Omega_r$
  say. Let $(a_{i1},\dots, a_{id})$ be an element of the orbit
  $\Omega_i$, $1\le i\le r$. We can apply Construction~\ref{const-subd} to $G_i=\langle a_{i1},\dots, a_{id}\rangle=G$, $1\le i\le r$, to obtain $\hat G=\langle a_1,\dots, a_d\rangle$. If $G$ is a primitive group of
  type~1 and socle~$N$, then the socle of $\hat G$ is the  product
  of all faithful and irreducible modules for~$G/N$ such that the
  corresponding primitive group is isomorphic to~$G$.
\end{construction}
\begin{example}\label{ex-many-crowns}
  There are three isomorphism classes of $2$-generated primitive
  groups of type~1 with socle of order~$8$, namely $G_1=[C_2^3]C_7$,
  $G_2=[C_2^3][C_7]C_3$ and $G_3=[C_2^3]{\GL_3(2)}$. We can construct
  $2$-generated groups 
  $\hat G_1$, $\hat G_2$, and $\hat G_3$ with Construction~\ref{const-hat}. By using Construction~\ref{const-subd}, we can construct a
  subdirect product $S$ of $\hat G_1$, $\hat G_2$, and $\hat G_3$ in such
  a way the generating pairs of all these three groups are identified. 

  An application of this construction with \textsf{GAP} \cite{GAP4-11-1} shows that
  \begin{align*}
    \abcr_8(\hat G_1)&=\abcr_8(\hat G_2)=16,&
    \abcr_8(\hat G_3)&=114,\\\abcr_7(\hat G_1)&=1,&\abcr_7(\hat
                                                    G_2)&=8,\\\abcr_3(\hat G_2)&=1,&
    \rk_{\operatorname{GL}_3(2)}(\hat
    G_3)&=57;
  \end{align*}
  the other values of the ranks and the numbers of abelian
  crowns are zero. We note that if $1\le i<j\le 3$, two chief factors of $S$ corresponding to chief factors of $\hat G_i$ and $\hat G_j$, respectively, cannot be $S$-isomorphic, because the groups of automorphisms induced by $S$ on the chief factors in the primitive groups $G_i$ and $G_j$ are not isomorphic. Therefore $\abcr_3(S)=1$, $\abcr_7(S)=9$, $\abcr_8(S)=146$, $\rk_{\operatorname{GL}_3(2)}(S)=57$; the other
    values are zero (we note that, in fact, $S$ coincides with the direct product $\hat G_1\times \hat G_2\times \hat G_3$). Since the indices of the maximal subgroups of
    $\operatorname{GL}_3(2)$ are $7$ and $8$ (see for instance
    \cite{Atlas85}), we conclude that $\rks_7(S)=\rks_8(S)=57$ and
    $\rko_{168}(S)=\rkm_{168}(S)=57$. The crowns of chief factors of order~$8$ in $\hat G_1$ are minimal normal subgroups, in $\hat G_2$ are products of three minimal normal subgroups, while in $\hat G_3$ are products of two minimal normal subgroups. The crowns of chief factors of order~$7$ in $\hat G_1$ and $\hat G_2$ coincide with the corresponding chief factors. There is a unique crown composed of two central chief factors of order~$3$. Hence $S$ has $16+16\cdot 3+114\cdot 2=292$ chief factors of order~$8$, $8$ chief factors of order~$7$, $2$ chief factors of order~$3$, and $57$ chief factors isomorphic to~$\GL_3(2)$.

    By
    Theorem~\ref{th-A} we obtain the bound
    \begin{align*}
      \invnu(S)\le \max\bigl\{&d+2.02+\log_32+\log_3\abcr_3(S),\\
      &d+2.02+\log_72+\log_7\abcr_7(S),\\
      &d+2.02+\log_82+\log_8\abcr_8(S),\\
      &4.02+\log_72+\log_7\rks_7(S),\\
                              &4.02+\log_82+\log_8\rks_8(S),\\
      &4.02+ \log_{168}\rkm_{168}(G)+\log_{168}\rko_{168}(G)
\bigr\},
    \end{align*}
    that is,
    \begin{align*}
      \invnu(S)\le \max\{&4.02+\log_32+\log_31,\\
      &4.02+\log_72+\log_79,\\
      &4.02+\log_82+\log_8146,\\
      &4.02+\log_72+\log_757,\\
      &4.02+\log_82+\log_857,\\
      &4.02+2\log_{168}57\}\le 6.75.
    \end{align*}

    The arguments of Lubotzky would take into account the $2$ abelian chief factors of order~$3$ to obtain the upper bound for $\maxi_3(S)$; the $8$ abelian chief factors of order~$7$ and the $57$ non-abelian chief factors isomorphic to~$\GL_3(2)$ to obtain the upper bound for $\maxi_7(S)$; the $292$ abelian chief factors of order~$8$  and the $57$ chief factors isomorphic to~$\GL_3(2)$ to obtain the upper bound for $\maxi_8(S)$, and the $57$ chief factors isomorphic to~$\GL_3(2)$ and order $168$ to obtain the upper bound for $\maxi_{168}(S)$. We obtain the following bounds:
    \begin{align*}
      \maxi_3(S)&\le ((1/2)(0+1)\cdot 0+2\cdot 3^2)3^2=162,\\
      \maxi_7(S)&\le             ((1/2)(57+1)\cdot 57 + 8\cdot 7^2)7^2=100\,205,\\
      \maxi_8(S)&\le ((1/2)(57+1)\cdot 57 + 292\cdot 8^2)8^2=1\,301\,824,\\
      \maxi_{168}(S)&\le ((1/2)(57+1)\cdot 57+0\cdot 168^2)168^2=46\,654\,272
    \end{align*}
    and so we conclude that
    \begin{align*}
      \invnu(S)&\le 2.02+\max\{\log_3(162), \log_7(100\,205),\\
               &\qquad\qquad\qquad\qquad\log_8(1\,301\,824), \log_{168}(46\,654\,272)\\
      &=2.02+\log_8(1\,301\,824)\le 8.791.
    \end{align*}
    The bounds obtained for the numbers of maximal subgroups of each index with the application of our Theorem~\ref{th-mn-long} are:
    \begin{align*}
      \maxi_3(S)&\le (3^2-1)\abcr_3(S)=8,\\
      \maxi_7(S)&\le (7^2-1)\abcr_7(S)+7^2\rks_7(S)=48\cdot 9+49\cdot 57=3\,225,\\
      \maxi_8(S)&\le (8^2-1)\abcr_8(S)+8^2\rks_8(S)=63\cdot 146+64\cdot 57=12\,846,\\
      \maxi_{168}(S)&\le 168^2\left(\frac{\rkm_{168}(S)(\rko_{168}(S)-s_{168})}{2}\right)\\&=168^2\cdot \left(\frac{57\cdot 56}{2}\right)=45\,045\,504.
    \end{align*}
    In fact, Theorem~\ref{th-nd-abelian-crown} gives the exact value $\maxi_3(S)=(3^2-1)/(3-1)=4$.
  \end{example}

\begin{remark}\label{rem-many-crowns}
  The procedure of Example~\ref{ex-many-crowns} can be followed with
  all isomorphism classes of $d$-generated monolitic primitive groups with a
  given socle to obtain a group with all possible crowns of
  non-Frattini chief factors of a given order. It is enough to consider all isomorphism classes of $d$-generated monolitic primitive groups whose socle has a given order, to use Construction~\ref{const-hat} with each of them, and then to use Construction~\ref{const-subd} with the obtained groups.
\end{remark}

\section*{Acknowledgements}
These results are part of the R+D+i project supported by the Grant 
PGC2018-095140-B-I00, funded by MCIN/AEI/10.13039/501100011033 and by ``ERDF A way of making Europe'', as well as by the Grant PROMETEO/2017/057
funded by GVA/10.13039/501100003359, and partially supported by the Grant E22 20R, funded by Departamento de Ciencia, Universidades y Sociedad del
Conocimiento, Gobierno de Arag\'on/10.13039/501100010067.

\bibliographystyle{alpha}
\bibliography{bibgroup}
\end{document}